\documentclass[11pt,reqno]{amsart}
\usepackage{amssymb}
\usepackage{mathptmx}
\usepackage{mathtools}
\usepackage{mathrsfs}
\usepackage[all]{xy}
\usepackage{stmaryrd}
\usepackage{fancyhdr}
\usepackage{hyperref}
\usepackage{mathrsfs}
\usepackage{tikz-cd}
\usepackage{cite}
\usepackage{amsmath,amsfonts}
\usepackage{graphicx}
\usepackage{wrapfig}
\usepackage{array}
\usepackage{amsthm}
\usepackage[foot]{amsaddr}
\usepackage[left=1.2in, right=1.2in, top=1in, bottom=1in]{geometry}
\usepackage{etoolbox}
\usepackage{hyperref}
\patchcmd{\section}{\scshape}{\bfseries}{}{}
\makeatletter
\renewcommand{\@secnumfont}{\bfseries}
\makeatother
\xyoption{all}
\usepackage{secdot}

\theoremstyle{plain}

\DeclareMathOperator{\K}{\mathcal{K}^\times}
\DeclareMathOperator{\Cart}{Cart}
\DeclareMathOperator{\CaCl}{CaCl}

\DeclareMathOperator{\Hom}{Hom}

\newcommand\HH{\mathrm{H}}

\DeclareMathOperator{\Spec}{Spec}
\DeclareMathOperator{\Frac}{Frac}

\newcommand\cHH{\check{\mathrm{H}}}

\DeclareMathOperator{\Pic}{Pic}

\input xy
\xyoption{all}

\theoremstyle{definition}
\newtheorem{mydef}{\textbf{Definition}}[section]
\newtheorem{myeg}[mydef]{\textbf{Example}}

\newtheorem{thm}[mydef]{\textbf{Theorem}}

\newtheorem{rmk}[mydef]{\textbf{Remark}}
\newtheorem{cond}[mydef]{\textbf{Condition}}
\theoremstyle{plain}

\newtheorem*{nothma}{\textbf{Theorem A}}
\newtheorem*{nothmb}{\textbf{Theorem B}}

\newtheorem{mytheorem}[mydef]{\textbf{Theorem}}
\newtheorem{lem}[mydef]{\textbf{Lemma}}
\newtheorem{prop}[mydef]{\textbf{Proposition}}

\newtheorem{cor}[mydef]{\textbf{Corollary}}

\def\rma{{\mathbb{T}}}

\begin{document}

\title{Picard groups for tropical toric schemes}

\author{Jaiung Jun$^1$}
\address{$^1$Binghamton University, Binghamton, NY 13902, USA, e-mail: jjun@math.binghamton.edu}

\author{Kalina Mincheva$^2$}
\address{$^2$Yale University, New Haven, CT 06511, USA, e-mail: kalina.mincheva@yale.edu, phone: (443)240-9747}

\author{Jeffrey Tolliver$^3$}
\address{$^3$Akuna Capital, Cambridge, MA 02139, USA, e-mail: jeff.tolli@gmail.com}



\dedicatory{}

\maketitle

\begin{abstract}
From any monoid scheme $X$ (also known as an $\mathbb{F}_1$-scheme) one can pass to a semiring scheme (a generalization of a tropical scheme) $X_S$ by scalar extension to an idempotent semifield $S$. We prove that for a given irreducible monoid scheme $X$ (satisfying some mild conditions) and an idempotent semifield $S$, the Picard group $\Pic(X)$ of $X$ is stable under scalar extension to $S$ (and in fact to any field $K$). In other words, we show that the groups $\Pic(X)$ and $\Pic(X_S)$ (and $\Pic(X_K)$) are isomorphic. In particular, if $X_\mathbb{C}$ is a toric variety, then $\Pic(X)$ is the same as the Picard group of the associated tropical scheme.
The Picard groups can be computed by considering the correct sheaf cohomology groups. We also define the group $\CaCl(X_S)$ of Cartier divisors modulo principal Cartier divisors for a cancellative semiring scheme $X_S$ and prove that $\CaCl(X_S)$ is isomorphic to $\Pic(X_S)$. 
\end{abstract}

\section{Introduction}
In recent years, there has been a growing interest in developing a notion of algebraic geometry over more general algebraic structures than commutative rings or fields. The search for such a theory is interesting in its own right. The current work, however, relates to two actively growing sub-fields of that study. \\
The first one is motivated by the search for ``absolute geometry" (commonly known as $\mathbb{F}_1$-geometry or algebraic geometry in characteristic one) which is first mentioned by J.~Tits in \cite{tits1956analogues}; Tits first hints at the existence of a mysterious field of ``characteristic one" by observing a degenerating case of an incidence geometry $\Gamma(K)$ associated to a Chevalley group $G(K)$ over a field $K$; when $K=\mathbb{F}_q$ (a field with $q$ elements), as $q \to 1$, the algebraic structure of $K$ completely degenerates, unlike the geometric structure of $\Gamma(K)$. This is why Tits suggests that $\lim_{q \to 1}\Gamma(K)$ should be a geometry over the field of characteristic one. \\
In \cite{manin1995lectures}, Y.~Manin considers the field of characteristic one from the completely different perspective: for developing a geometric approach to the Riemann hypothesis. Shortly after, in \cite{soule2004varietes}, C.~Soul\'{e} first introduced a notion of algebraic geometry over the field $\mathbb{F}_1$ with one element. Since then A.~Connes and C.~Consani have worked to find a geometric framework which could allow one to adapt the Weil proof of the Riemann hypothesis for function fields to the Riemann zeta function (cf. \cite{con4}, \cite{con1}, \cite{con2}, \cite{connes2014arithmetic}, \cite{connes2016geometry}, \cite{connes2017homological}). \\
The second field to which this work contributes is a new branch of algebraic geometry called tropical geometry. It studies an algebraic variety $X$ over a valued field $k$ by means of its ``combinatorial shadow", called the (set-theoretic) tropicalization of $X$ and denoted $trop(X)$. This is a degeneration of the original variety to a polyhedral complex obtained from $X$ and a valuation on $k$. The combinatorial shadow retains a lot of information about the original variety and encodes some of its invariants. Algebraically, $trop(X)$ is described by polynomials in an idempotent semiring, which is a more general object than a ring -- a semiring satisfies the same axioms that a ring does, except invertibility of addition. However, the tropical variety $trop(X)$ has no scheme structure. In \cite{giansiracusa2016equations}, J.~Giansiracusa and N.~Giansiracusa combine $\mathbb{F}_1$-geometry and tropical geometry in an elegant way to introduce a notion of tropical schemes. This is an enrichment to the structure of $trop(X)$, since a tropical variety is seen as the set of geometric points of the corresponding tropical scheme. The tropical scheme structure is encoded in a ``bend congruence" (cf. \cite{giansiracusa2016equations}) or equivalently in a ``tropical ideal" or a tower of valuated matroids (cf. \cite{maclagan2014tropical}, \cite{maclagan2016tropical}). \\
A natural question and central motivation for this work, is whether one can find a scheme-theoretic tropical Riemann-Roch theorem using the notions from J.~Giansiracusa and N.~Giansiracusa's theory. \\
The problem of finding such an analogue for tropical varieties, which one obtains through set-theoretic tropicalization, has already received a lot of attention. In particular it has been solved in the case of tropical curves. We recall that a tropical curve (the set-theoretical tropicalization of a curve) is a connected metric graph. One can define divisors, linear equivalence and genus of a graph in a highly combinatorial way. 
With these notions M.~Baker and S.~Norine \cite{baker2007riemann} prove a Riemann-Roch theorem for finite graphs while G.~Mikhalkin and I.~ Zharkov \cite{MikhZharkov} and later A.~Gathmann and M.~Kerber \cite{gathmanKerber} solve the problem for metric graphs. Later the Riemann-Roch problem is revisited by O.~Amini and L.~Caporaso \cite{aminiCaporaso} who consider weighted graphs. A generalization of the work of M.~Baker and S.~Norine to higher dimensions has been carried out by D.~Cartwright in \cite{dustinTropComplex2} and \cite{dustinTropComplex}.\\
In \cite{connes2016geometry} A.~Connes and C.~Consani prove a Riemann-Roch statement for the points of a certain Grothendieck topos over (the image of) $\Spec(\mathbb{Z})$, which can be thought of as tropical elliptic curves. An important ingredient to the project of A.~Connes and C.~Consani to attack the Riemann Hypothesis is to develop an adequate version of the Riemann-Roch theorem for geometry defined over idempotent semifields. Notably, there are several fundamental differences between their statement and the tropical Riemann-Roch of M.~Baker and S.~Norine.\\
Recently, in \cite{DhruvlogPic} T.~Foster, D.~Ranganathan, M.~Talpo and M.~Ulirsch investigate the logarithmic Picard group (which is a quotient of the algebraic Picard group by lifts of relations on the tropical curve) and solve the Riemann-Roch problem for logarithmic curves (which are metrized curve complexes). \\
The solution of a scheme-theoretic tropical Riemann-Roch problem requires several ingredients, such as a proper framework for ``tropical sheaf cohomology" and a notion of divisors on tropical schemes, and most importantly ranks of divisors. In this note, we investigate Picard groups of tropical (toric) schemes as the first step towards building scheme-theoretic tropical divisor theory and a Riemann-Roch theorem. To do that, we look at an $\mathbb{F}_1$-model of the tropical scheme, i.e., a monoid scheme. \\
Monoid schemes are related to tropical (more generally semiring) schemes and usual schemes via scalar extension. More precisely, if $X$ is a monoid scheme, say $X = \Spec M$ for a monoid $M$, then $X_K = \Spec K[M]$ is a scheme if $K$ is a field and a semiring scheme if $K$ is a semifield. Note that $X_{\mathbb{C}}$ is a toric variety if $M$ is integral (cancellative), finitely generated, saturated and torsion free and $X_K$ is a tropical scheme (the scheme theoretic tropicalization of $X_{\mathbb{C}}$) if $K$ is the semifield of tropical numbers. When $K$ is a field, the relation between $\Pic(X)$ and $\Pic(X_K)$ has been extensively investigated by J.~Flores and C.~Weibel in \cite{flores2014picard}. They show that in this case $\Pic(X)$ is isomorphic to $\Pic(X_K)$. \\
In this paper, our main interest lies in the tropical setting, i.e., the case when $K$ is an idempotent semifield. We investigate the relation between the Picard group $\Pic(X)$ of a monoid scheme $X$ and $\Pic(X_K)$ for an idempotent semifield $K$. More precisely, when $X_K$ is the lift of an $\mathbb{F}_1$-model $X$ (with an open cover satisfying a mild finiteness condition), we prove that the Picard group is stable under scalar extension. 

\begin{nothma} Let $X$ be an irreducible monoid scheme and $K$ be an idempotent semifield. Suppose that $X$ has an open cover satisfying Condition \ref{condition: condition for open cover}. Then we have that $$\Pic(X) = \Pic(X_K) = \HH^1(X, \mathcal{O}_X^{\times}).$$
\end{nothma}


\begin{rmk}
This theorem provides an alternative proof to Theorem 6.6 in \cite{flores2014picard}, giving a new proof to Fulton's result on the Picard group of a toric variety (cf. Remark~\ref{rmk:Fulton}).
\end{rmk}

Recall that a cancellative semiring scheme $X_K$ over an idempotent semifield $K$ is a semiring scheme such that for each open subset $U$ of $X_K$, the semiring $\mathcal{O}_{X_K}(U)$ of sections is cancellative. We construct the group of Cartier divisors and modulo principal Cartier divisors, which we denote by $\CaCl(X_K)$. These are the naive principal Cartier divisors - the ones defined by a single global section. Then we can show that

\begin{nothmb} 
Let $X_K$ be a cancellative semiring scheme over an idempotent semifield $K$. Then $\CaCl(X_K)$ is isomorphic to $\Pic(X_K)$. 
\end{nothmb}

We remark that the Picard group of a curve $C$ is rarely equal to the Picard group of its set-theoretic tropicalization (non-Archimedian skeleton). However, the skeleton of the Jacobian is canonically isomorphic to the Jacobian of the skeleton as principally polarized tropical abelian varieties as shown in \cite{BakerRabinoff}. Moreover, in the case then $C$ is a smooth toric curve, the Picard group of $C$ is equal to the Picard group of the non-Archimedian skeleton. \\


Since the tropical toric schemes considered in this paper have an $\mathbb{F}_1$ model, it will be interesting to understand the relation of the results presented here to the theory of cohomology over $\mathbb{F}_{1^2}$ in \cite{flores2017vcech}. This will allow to show that other invariants of toric varieties also live over $\mathbb{F}_1$.\\

\textbf{Acknowledgments}
The work on this note was started at the Oberwolfach workshop ID: 1525 on Non-commutative Geometry. The authors are grateful for the Institute's hospitality and would like to thank the organizers for providing the opportunity. K.M would also like to thank Dhruv Ranganathan for many helpful conversations. We are thankful to Oliver Lorscheid for suggesting ways to strengthen the results of this note. We would also like to thank the anonymous referees for their valuable comments and suggestions. 
\vspace{0.1cm}

\section{Preliminaries}
In this section, we review some basic definitions and properties of monoid and semiring schemes. We also recall the notion of Picard groups for monoid schemes developed in \cite{flores2014picard} and for semiring schemes introduced in \cite{jun2015cech}. 

\subsection{Picard groups for monoid schemes}

In what follows, by a monoid we always mean a commutative monoid $M$ with an absorbing element $0_M$, i.e., $0_M\cdot m=0_M$ for all $m \in M$. Note that if $M$ is a monoid without an absorbing element, one can always embed $M$ into $M_0=M\cup\{0_M\}$ by letting $0_M\cdot m=0_M$ for all $m \in M$. 

\begin{rmk}
We will use the term ``monoid schemes" instead of ``$\mathbb{F}_1$-schemes" to emphasize that we are employing the minimalistic definition of $\mathbb{F}_1$-schemes based on monoids following A.~Deitmar \cite{Deitmar}, instead of any of the more general constructions that exist in the literature (cf. \cite{con1} or \cite{oliver1}).
\end{rmk}

We recall some important notions which will be used throughout the paper. For the details, we refer the reader to \cite{chu2012sheaves}, \cite{cortinas2015toric}, \cite{Deitmar}, \cite{deitmar2008f1}.

\begin{mydef}\cite[\S1.2 and \S 2.2]{Deitmar}
Let $M$ be a monoid.
\begin{enumerate}
\item
An \emph{ideal} $I$ is a nonempty subset of $M$ such that $MI \subseteq I$. In particular, it contains the absorbing element $0_M$. A proper ideal $I$ is said to be \emph{prime} if $M\backslash I$ is a multiplicative nonempty subset of $M$.
\item
A \emph{maximal ideal} of $M$ is a proper ideal which is not contained in any other proper ideal.
\item
The \emph{prime spectrum} $\Spec M$ of a monoid $M$ is the set of all prime ideals of $M$ equipped with the Zariski topology. A basis for the topology is given by the collection of open sets 
\[
D(f)=\{\mathfrak{p} \in \Spec M \mid f \not\in \mathfrak{p}\},
\]
for all $f \in M$.
\end{enumerate}
\end{mydef}

One can mimic the construction of the structure sheaf on a scheme to define the structure sheaf (of monoids) on the prime spectrum $\Spec M$ of a monoid $M$. The prime spectrum $\Spec M$ together with a structure sheaf is called an \emph{affine monoid scheme}. A locally monoidal space is a topological space together with a sheaf of monoids. A \emph{monoid scheme} is a monoidal space which is locally isomorphic to an affine monoid scheme. As in the classical case, we define a morphism between two locally monoidal spaces $(X,\mathcal{O}_X)$ and $(Y,\mathcal{O}_Y)$ to be a pair $(f,f^\#)$ consisting of a continuous map $f:X \to Y$ and a morphism $f^\#:\mathcal{O}_Y \to f_*\mathcal{O}_X$ of sheaves on $Y$ which is local, i.e., at each point $x \in X$, the induced map $f^\#_x:\mathcal{O}_{Y,f(x)} \to \mathcal{O}_{X,x}$ on stalks is a local homomorphism of monoids. Morphisms of monoid schemes are morphisms of locally monoidal spaces. As in the case of schemes, we call a monoid scheme \emph{irreducible} if the underlying topological space is irreducible. 

\begin{rmk}
\begin{enumerate}
\item
Let $M$ be a monoid. Then $M$ has a unique maximal ideal $\mathfrak{m}=M\backslash M^\times$.
\item 
The category of affine monoid schemes is equivalent to the opposite of the category of monoids. A monoid scheme, in this case, is a functor which is locally representable by monoids. In other words, one can understand a monoid scheme as a \emph{functor of points}. For details, see \cite{pena2009mapping}.
\end{enumerate}
\end{rmk}

Next, we briefly recall the definition of invertible sheaves on a monoid scheme $X$. We refer the readers to \cite{chu2012sheaves} and \cite{flores2014picard} for details.

\begin{mydef}\cite[\S 5]{flores2014picard}\label{def: inv_sheaf}
Let $M$ be a monoid and $X$ be a monoid scheme.
\begin{enumerate}
\item
By an \emph{$M$-set}, we mean a set with an $M$-action.
\item
An \emph{invertible sheaf} $\mathcal{L}$ on $X$ is a sheaf of $\mathcal{O}_X$-sets which is locally isomorphic to $\mathcal{O}_X$.
\end{enumerate}
\end{mydef}
Let $M$ be a monoid and $A,B$ be $M$-sets. One may define the \emph{tensor product} $A\otimes_M B$. Furthermore, if $A$ and $B$ are monoids, $A\otimes_M B$ becomes a monoid in a canonical way. See, \cite[\S 2.2, 3.2]{chu2012sheaves} for details. 

\begin{rmk}
We would like to warn the reader that in the literature the same terminology is used to denote different things. In \cite{flores2014picard}, the authors use the term ``smash product" for tensor product of monoids, whereas in \cite{chu2012sheaves}, the authors use the term ``smash product" only for tensor product of monoids over $\mathbb{F}_1$, the initial object in the category of monoids. We will use the term tensor product to stay compatible with the language of schemes and semiring schemes. 
\end{rmk}

Let $X$ be a monoid scheme and let $\Pic(X)$ be the set of the isomorphism classes of invertible sheaves on $X$. Suppose that $\mathcal{L}_1$ and $\mathcal{L}_2$ are invertible sheaves on $X$ as in Definition~\ref{def: inv_sheaf}. The tensor product $\mathcal{L}_1\otimes_{\mathcal{O}_X}\mathcal{L}_2$ is the sheafification of the presheaf (of monoids) sending an open subset $U$ of $X$ to $\mathcal{L}_1(U)\otimes_{\mathcal{O}_X(U)}\mathcal{L}_2(U)$. It is well--known that $\mathcal{L}_1\otimes_{\mathcal{O}_X}\mathcal{L}_2$ is an invertible sheaf on $X$. Also, as in the classical case, the sheafificaiton $\mathcal{L}_1^{-1}$ of the presheaf (of monoids) sending an open subset $U$ of $X$ to $\Hom_{\mathcal{O}_X(U)}(\mathcal{L}_1(U),\mathcal{O}_X(U))$ becomes an inverse of $\mathcal{L}_1$ with respect to the tensor product in the sense that $\mathcal{L}_1\otimes_{\mathcal{O}_X} \mathcal{L}_1^{-1} \simeq \mathcal{O}_X$. Therefore, $\Pic(X)$ is a group.\\
One can use the classical argument to prove that $\Pic(X)\simeq \HH^1(X,\mathcal{O}_X^{\times})$ (for instance, see \cite[Lemma 5.3.]{flores2014picard}). Recall that for any topological space $X$ and a sheaf $\mathcal{F}$ of abelian groups on $X$, sheaf cohomology $\HH^i(X,\mathcal{F})$ and \v{C}ech cohomology $\cHH^i(X,\mathcal{F})$ agree for $i=0,1$. Therefore, we have
\begin{equation}\label{equation: pic=derived=cech}
\Pic(X)\simeq \HH^1(X,\mathcal{O}_X^{\times}) \simeq \cHH^1(X,\mathcal{O}_X^{\times}). 
\end{equation}

\subsection{Picard groups for semiring schemes}
A \emph{semiring} is a set (with two binary operations - addition and multiplication) that satisfies the same axioms that a ring does, except invertibility of addition. In this paper by a semiring we mean a commutative semiring with at least two distinct elements (a multiplicative identity and an additive identity). In particular, a semiring is a commutative monoid with respect to both operations. A \emph{semifield} is a semiring in which every non-zero element has a multiplicative inverse. For $S$ a semifield, an $S$-algebra is a morphism of semirings $S \rightarrow A$. A semiring $A$ is \emph{idempotent} if $a+a = a$, for all elements $a \in A$. An example of an (idempotent) semifield is the \emph{tropical semifield} which we denote by $\rma$. It is defined on the set $\mathbb{R}\cup \{-\infty\} $, with operations maximum and addition. 

\begin{mydef}
Let $A$ and $B$ be semirings. Then a function $f: A \rightarrow B$ is a morphism of semirings if and only if:
\begin{enumerate}
    \item $f(0_A) = 0_{B}$,
    \item $f(1_A) = 1_{B}$, 
    \item $f(a+a')=f(a)+f(a')$, and $f(aa')=f(a)f(a')$, $\forall a,a'\in A$.
\end{enumerate}
\end{mydef}

\begin{mydef}
Let $A$ be a semiring.
\begin{enumerate}
\item
An \emph{ideal} $I$ of $A$ is an additive submonoid $I$ of $A$ such that $AI \subseteq I$.
\item
An ideal $I$ is said to be \emph{prime} if $I$ is proper and if $ab \in I$ then $a \in I$ or $b \in I$.
\item
A proper ideal $I$ is \emph{maximal} if the only ideal strictly containing $I$ is $A$.
\end{enumerate}
\end{mydef}

There are several definitions of semiring schemes in the literature. One can find a complete list of the proposed structures and the relations between them in \cite{oliver4}. We present the following definition, originally introduced in \cite{giansiracusa2016equations}, which we will use in this paper.

\begin{mydef}
Let $A$ be a semiring algebra and $X=\Spec A$ be the set of all prime ideals of $A$. We endow $X$ with the Zariski topology. The topology on $X$ is generated by the sets of the form $D(f)=\{\mathfrak{p} \in X \mid f \not\in \mathfrak{p}\}$ for all $f \in A$.
\end{mydef}

Let $A$ be a semiring and $S$ be a multiplicative subset of $A$. We recall the construction of the \emph{localization} $S^{-1}A$ of $A$ at $S$ from \cite[\S11]{semibook}. Let $M=A\times S$. We impose an equivalence relation on $M$ in such a way that $(a,s) \sim (a',s')$ if and only if $\exists t \in S$ such that $t a s'=t a's$. The underlying set of $S^{-1}A$ is the set of equivalence classes of $M$ under $\sim$. We let $\frac{a}{s}$ be the equivalence class of $(a,s)$. Then one can define the following binary operations $+$ and $\cdot$ on $S^{-1}A$:
\[
\frac{a}{s}+\frac{a'}{s'}=\frac{as'+sa'}{ss'}, \quad \frac{a}{s}\cdot \frac{a'}{s'}=\frac{aa'}{ss'}.
\]
It is well--known that $S^{-1}A$ is a semiring with the above operations. Furthermore, we have a canonical homomorphism $S^{-1}:A \to S^{-1}A$ sending $a$ to $\frac{a}{1}$. When $S=A-\mathfrak{p}$ for some prime ideal $\mathfrak{p}$, we denote the localization $S^{-1}A$ by $A_\mathfrak{p}$.\\

Let $A$ be a semiring algebra and $X=\Spec A$ be the prime spectrum of $A$. For each Zariski open subset $U$ of $X$, we define the following set:
\[
\mathcal{O}_X(U)=\{s:U \to \prod_{\mathfrak{p} \in U} A_\mathfrak{p}\},
\]
where $s$ is a function such that $s(\mathfrak{p}) \in A_\mathfrak{p}$ and $s$ is locally representable by fractions. One can easily see that $\mathcal{O}_X$ is a sheaf of semirings on $X$. An \emph{affine semiring scheme} is the prime spectrum $X=\Spec A$ equipped with a structure sheaf $\mathcal{O}_X$. Next, by directly generalizing the classical notion of locally ringed spaces, one can define a locally semiringed space, as a topological space with a sheaf of semirings such that at each point, the stalk has a unique maximal ideal. A \emph{semiring scheme} is a locally semiringed space which is locally isomorphic to an affine semiring scheme. 


A special case of the semiring schemes proposed by J. Giansiracusa and N. Giansiracusa are the $\mathbb{T}$-schemes. They are locally isomorphic to the prime spectrum of a quotient of the polynomial semiring over the tropical semifield (denoted $\rma[x_1, \dots, x_n]$) by a particular equivalence relation, called a ``bend congruence''. In this paper we refer to these schemes as ``tropical schemes''. We point out that in \cite{maclagan2014tropical} and \cite{maclagan2016tropical} the term ``tropical schemes'' is reserved for $\mathbb{T}$-schemes defined by the bend relations of special ideals, called tropical ideals.

We note that every scheme is a semiring scheme, but never a tropical scheme. The reason is that the structure sheaf of a tropical scheme is a sheaf of additively idempotent semirings which are never rings.  

One can extend the familiar notions of invertible sheaves and Picard group of schemes or monoid schemes to the semiring schemes setting. In fact, \v{C}ech cohomology for semiring schemes is introduced in \cite{jun2015cech} and the following is proved. 

\begin{thm}\cite{jun2015cech}
Let $X$ be a semiring scheme. Then $\Pic(X)$ is a group and can be computed via $\HH^1(X,\mathcal{O}_X^{\times})$.
\end{thm}

We note that $\mathcal{O}_X^{\times}$ is a sheaf of abelian groups and thus we can define $\HH^1(X,\mathcal{O}_X^{\times})$ in the usual way and $\HH^1(X,\mathcal{O}_X^{\times})=\cHH^1(X,\mathcal{O}_X^{\times})$ as in Equation~\eqref{equation: pic=derived=cech}.  

\begin{myeg}
In \cite{jun2015cech}, the author also proves that for a projective space $\mathbb{P}^n_S$ over an idempotent semifield $S$, one obtains that $\Pic(\mathbb{P}^n_S)\simeq \mathbb{Z}$, as in the classical case. In \cite{flores2014picard}, a similar result is proven for a projective space over a monoid. These two results motivate (among others) the authors of the current note to study relations among Picard groups of schemes, monoid schemes, and semiring schemes under ``scalar extensions". 
\end{myeg}

\section{Picard groups of tropical toric schemes}
In this section, we prove the main result which states that the Picard group $\Pic(X)$ of an irreducible monoid scheme $X$ (with some mild conditions) is stable under scalar extension to an idempotent semifield. Let us first recall the definition of scalar extension of a monoid scheme to a field or an idempotent semifield.\\

Let $X$ be a monoid scheme and $K$ a field or an idempotent semifield. Suppose that $X$ is affine, i.e., $X=\Spec M$, for some monoid $M$. Then \emph{the scalar extension} is defined as: 
\[
X_K=X\times_{\Spec\mathbb{F}_1}\Spec K=\Spec K[M],
\] 
where $K[M]$ is a monoid semiring (when $K$ is an idempotent semifield) or a monoid ring (when $K$ is a field). This construction can be globalized to define the base change functors from monoid schemes to semiring schemes (or schemes). For the base change functors from monoid schemes to schemes, see \cite{Deitmar} and the case for semiring schemes, see \cite[\S 3]{oliver1} or \cite[\S 3.2.]{giansiracusa2016equations}. In our case, we will use the following lemma:

\begin{lem}\label{lem: open embedding}
	Let $X$ be an irreducible monoid scheme, $U$ be an open subset of $X$, and $K$ be an idempotent semifield. Then \[
	U_K=U \times_{\Spec\mathbb{F}_1} \Spec K
	\]
	is an open subset of $X_K$. 
\end{lem}
\begin{proof}
	We may assume that $X$ is affine, say $X=\Spec M$, and $U=D(f)$ for some $f \in M$. Then one can easily check that
	\[
	U_K=\{\mathfrak{p} \in\Spec K[M]\mid 1_M\cdot f \not\in \mathfrak{p}\}. 
	\]
	In other words, $U_K$ is nothing but $D(f)$ when $f$ is considered as an element of $K[M]$. 
\end{proof}

\begin{rmk}
In fact, the scalar extensions can be defined for any semiring $S$ by considering the monoid semiring $S[M]$ for a given monoid $M$. But in this paper, we only consider the case when $S$ is an idempotent semifield. 
\end{rmk}
\begin{rmk}\label{rmk: toricDeitmar}
In \cite{cortinas2015toric}, G.~Corti\~{n}as, C.~Haesemeyer, M.~Walker, and C.~Weibel prove that from a toric monoid scheme $X$ (a monoid scheme of finite type which is separated, connected, torsion-free, and normal), after scalar extension to a field $k$, one obtains a toric variety $X_k$. Conversely, the authors also show that from a toric variety $X_k$ associated to a fan $\Delta$, one can always construct a toric monoid scheme $X=X(\Delta)$ (from a fan $\Delta)$ in such a way that $X_k$ is the scalar extension of $X$ to $k$. 
Moreover, they show that there exists a faithful functor from fans to a toric monoid schemes (cf. Theorem~4.4 \cite{cortinas2015toric}).
\end{rmk}

Since the Picard group $\Pic(X)$ of a semiring scheme $X$ can be computed by the cohomology group with values in the sheaf $\mathcal{O}_X^{\times}$, we first need to understand $K[M]^\times$, the group of multiplicatively invertible elements of $K[M]$. \\

Recall that by a \emph{cancellative monoid}, we mean a monoid $M$ such that if $ab=ac$, for $a \in M-\{0_M\}$ and $b,c \in M$, then $b=c$. We insist on using this terminology, since passing to semirings integral (having no zero divisors, i.e., $ab=0$ implies either $a=0$ or $b=0$) and cancellative are two different notions. For example, the semiring of tropical polynomials does not contain zero divisors but is not cancellative (cf. Remark~\ref{rem: integral}).

\begin{prop}\label{KMunits}
Let $M$ be a cancellative monoid and $K$ be an idempotent semifield.  Then we have 
\[
K[M]^\times\cong K^\times \times M^\times.
\]
\end{prop}
\begin{proof}
For $x\in K[M]$, let $\phi(x)\in\mathbb{N}$ be the smallest natural number such that $x$ has a representation of the form 
\[
x=\displaystyle\sum_{k=1}^{\phi(x)}a_km_k, 
\]
where $a_k\in K$ and $m_k\in M$.  Such a representation has a minimal length if and only if each $a_k\neq 0$ and the elements $m_k$ are nonzero and distinct. Since $K$ is an idempotent semifield, two nonzero elements of $K$ cannot sum to zero and hence we have $\phi(x)\leq \phi(x+y)$ for all $x,y\in K[M]$.\\
Let $m\in M-\{0_M\}$ and $x\in K[M]$.  Write $x=\displaystyle\sum_{k=1}^{\phi(x)}a_km_k$, so we have that $mx=\displaystyle\sum_{k=1}^{\phi(x)}a_k(mm_k)$.  This is the shortest such expression for $mx$ because each $a_k\neq 0$ and $M$ is cancellative. Hence one has $\phi(mx)=\phi(x)$.  One also has $\phi(ax)=\phi(x)$ for $a\in K^\times$.\\
Now let $y\in K[M]$ be nonzero and $x\in K[M]$ be arbitrary.  Write $y=\displaystyle\sum_{k=1}^{\phi(y)}a_km_k$ with $a_k \neq 0$ for all $k=1,\dots,\phi(y)$. In particular, $a_1\in K^\times$ since $K$ is a semifield. Also, since $K$ is idempotent, we have that $y=a_1m_1+y$ and hence
\[ 
\phi(xy)=\phi(xy+a_1m_1x)\geq\phi(a_1m_1x)=\phi(x).
\]

Finally suppose $x\in K[M]^\times$ and let $y=x^{-1}$.  Then we have
\[
1=\phi(1)=\phi(xy)\geq \phi(x).
\]
Since $\phi(x)\leq 1$, we have $x=am$ for
$a\in K$ and $m\in M$.  Similarly $y=bm'$ for $b\in K$ and $m'\in M$.  One easily sees that $ab=1$ and $mm'=1$, and hence $x\in K^\times\times M^\times$.
\end{proof}

We can compute \v{C}ech cohomology given the existence of an appropriate open cover. The following result provides a link between $\Pic (X_K)$ and \v{C}ech cohomology of the sheaf $\mathcal{K}^\times \times {\mathcal{O}_X}^\times$ (of abelian groups) on $X$, where $\mathcal{K}^\times$ is the constant sheaf on $X$ associated to the abelian group $K^\times$. 


In what follows, we will assume the following condition for an irreducible monoid scheme. 

\begin{cond}\label{condition: condition for open cover}
Let $X$ be an irreducible monoid scheme. Suppose that $X$ has an open affine cover $\mathcal{U}=\{U_\alpha\}$ such that any finite intersection of the sets $U_\alpha$ is isomorphic to the prime spectrum of a cancellative monoid. 
\end{cond}

\begin{rmk}
A monoid scheme $X$ is said to be separated if the diagonal map $\Delta:X \to X\times_{\Spec \mathbb{F}_1} X$ is a closed immersion. Note that the definition of a closed immersion slightly differs from the case of schemes. See, \cite[Definition 2.5., Remark 2.7.1., and $\S 4$]{cortinas2015toric} for more details. The above condition is not equivalent to the separatedness of monoid schemes. Following \cite[Corollary 3.8.]{cortinas2015toric}, for any separated monoid scheme $X$ and two affine open subschemes $U_1$ and $U_2$, the intersection $U_1 \cap U_2$ should be affine. But our condition insists further that such intersection should be given by a cancellative monoid. 
\end{rmk}

\begin{rmk}
Condition \ref{condition: condition for open cover} can be weakened as follows: one may only require the existence of an affine open cover $\mathcal{U}=\{U_\alpha\}$ such that any finite intersection of the sets $U_\alpha$ is isomorphic to a \emph{union} of prime spectra of cancellative monoids. But it helps to avoid a digression into technicalities, which would make the paper hard to read.
\end{rmk}

Let $X$ be an irreducible monoid scheme and $\mathcal{U}=\{U_i\}$ be an open cover of $X$. Let
\[
\mathcal{U}_K=\{U_i\times_{\Spec \mathbb{F}_1}\Spec K\}. 
\]
It follows from Lemma \ref{lem: open embedding} that $\mathcal{U}_K$ is an open cover of $X_K$ for an idempotent semifield $K$. \\

Now, under the assumption of Condition \ref{condition: condition for open cover}, we have the following. 

\begin{thm}\label{almostmainthm}
Let $X$ be an irreducible monoid scheme and $\mathcal{U}$ be an open cover of $X$ satisfying Condition \ref{condition: condition for open cover}. Let $K$ be an idempotent semifield, $X_K=X\times_{\Spec\mathbb{F}_1}\Spec K$, and  $\mathcal{U}_K=\{U_\alpha\times_{\Spec\mathbb{F}_1}\Spec K\}$. Then we have 
\[
\HH^1(\mathcal{U}_K,\mathcal{O}_{X_K}^\times)\cong \HH^1(\mathcal{U},\K\times\mathcal{O}_X^\times).
\]
\end{thm}
\begin{proof}
Let $U_\alpha^K=U_\alpha\times_{\Spec\mathbb{F}_1}\Spec K$.  Write $U_{\alpha_1\ldots\alpha_n}=U_{\alpha_1}\cap\ldots\cap U_{\alpha_n}$, and similarly for $U_{\alpha_1\ldots\alpha_n}^K$.  Let $\mathcal{F}=\K\times\mathcal{O}_X^\times$, and $\mathcal{G}=\mathcal{O}_{X_K}^\times$.  Note that by irreducibility of $X$, we have $\mathcal{F}(U)=K^\times \times (\mathcal{O}_X(U)^\times)$ for any open subset $U\subseteq X$.\\
Fix $\alpha_1,\ldots,\alpha_n$.  Then we have $U_{\alpha_1\ldots\alpha_n}\cong \Spec M$ and $U_{\alpha_1\ldots\alpha_n}^K\cong \Spec K[M]$ for some cancellative monoid $M$.  Then $\mathcal{F}(U_{\alpha_1\ldots\alpha_n})=K^\times\times M^\times \cong K[M]^\times=\mathcal{G}(U_{\alpha_1\ldots\alpha_n}^K)$, by Proposition \ref{KMunits}.  These isomorphisms induce isomorphisms of \v{C}ech cochains $\mathbf{C}^k(\mathcal{U},\mathcal{F})\cong\mathbf{C}^k(\mathcal{U}_K,\mathcal{G})$.  The result will follow if these isomorphisms are compatible with the differentials.  However, this reduces to checking that the maps $\mathcal{F}(U_{\alpha_1\ldots\alpha_n})\rightarrow\mathcal{F}(U_{\alpha_1\ldots\alpha_n\alpha_{n+1}})\rightarrow\mathcal{G}(U_{\alpha_1\ldots\alpha_n\alpha_{n+1}}^K)$ and $\mathcal{F}(U_{\alpha_1\ldots\alpha_n})\rightarrow\mathcal{G}(U_{\alpha_1\ldots\alpha_n}^K)\rightarrow \mathcal{G}(U_{\alpha_1\ldots\alpha_n\alpha_{n+1}}^K)$ agree, which is readily checked.
\end{proof}

Our next goal is to show that $\Pic X_K$ can be computed by using a cover that satisfies the assumption of Condition \ref{condition: condition for open cover}.

\begin{lem}\label{lem: idempotent semifield}
Let $K$ be an idempotent semifield. Suppose that $a,b \in K$ and $a \neq 0$ or $b \neq 0$, then $a+b \neq 0$. 
\end{lem}
\begin{proof}
Suppose that $a+b=0$. Since $K$ is idempotent we have 
\[
a+b=a+a+b=a+(a+b)=a+0=a.
\]
Similarly, we obtain that $a+b=b$. It follows that $a=b$ and hence $a+b=a=b$ as $K$ is idempotent. This implies that $a=b=0$, which contradicts the initial assumption. 
\end{proof}

\begin{lem}\label{coverofKM}
Let $K$ be an idempotent semifield and $M$ be a cancellative monoid. Let $\{U_\alpha\}$ be an open cover of $X=\Spec  K[M]$.  Then $U_\alpha=X$ for some $\alpha$.
\end{lem}
\begin{proof}
We may assume without loss of generality that each open set in the cover is nonempty. One can easily observe that $X$ contains a unique maximal ideal, namely, $\mathfrak{m}:=K[M] - K[M]^\times$; this directly follows from the proof of Proposition \ref{KMunits}. Since $\{U_\alpha\}$ is an open covering of $X$, $\mathfrak{m} \in U_\alpha$ for some $\alpha$. We may further assume that $U_\alpha$ is an affine open subset of $X$, say $U_\alpha=V(I)^c$ for some ideal $I$ of $K[M]$. In particular, we have that $I \not \subseteq \mathfrak{m}$. Now, let $\mathfrak{p}$ be any prime ideal of $K[M]$. Since $\mathfrak{m}$ is a unique maximal ideal, we have that $\mathfrak{p} \subseteq \mathfrak{m}$ and hence $I \not \subseteq \mathfrak{p}$, showing that $X=U_\alpha$.



\end{proof}

\begin{rmk}
A special case of Lemma \ref{coverofKM}, when $M$ is a free monoid generated by $n$ elements, is proved in \cite[Lemma 4.20]{jun2015cech} and referred to as a ``tropical partition of unity''. In fact, Lemma \ref{coverofKM} can be also proven by using the fact that $K[M]$ has the unique maximal ideal, namely $\mathfrak{m}=K[M]-K[M]^\times$ and apply the same argument as in Lemma \ref{lemma: monoidcovering}. Furthermore, as long as a semifield $K$ satisfies the condition in Lemma \ref{lem: idempotent semifield}, one can easily prove that Lemma \ref{coverofKM} holds. 
\end{rmk}

\begin{thm}\label{thm: affine vanishing}
Let $K$ be an idempotent semifield, $M$ be a cancellative monoid, and $X=\Spec K[M]$. Then $\Pic (X)=0$, and more generally, we have
\[
\HH^k(X,\mathcal{O}_{X}^\times)=0,\quad \textrm{for } k\geq 1.
 \]
\end{thm}
\begin{proof}

Let $\mathcal{U}=\{U_i\}$ be a covering of $X$ such that every other covering $\mathcal{V}=\{V_i\}$ of $X$ contains $\{U_i\}$ as a refinement. Then by definition of \v{C}ech cohomology $H^k(X,\mathcal F) = H^k(\mathcal{U},\mathcal F)$ for any abelian sheaf $\mathcal F$. By Lemma \ref{coverofKM}, this is the case for the covering $\{U_i\} = \{X\}$. Thus $H^k(X, \mathcal{O}_X^\times) = H^k(\mathcal{U},\mathcal{O}_X^\times)$. Now, one can easily compute $H^k(\mathcal{U},\mathcal{O}_X^\times)$ to prove the theorem.



\end{proof}

\begin{cor}\label{Kacylcic}
Let $X$ be an irreducible monoid scheme and $\mathcal{U}=\{U_\alpha\}$ be an affine open cover satisfying Condition \ref{condition: condition for open cover}. Let $K$ be an idempotent semifield and let $X_K=X\times_{\Spec\mathbb{F}_1}\Spec K$.  Let $\mathcal{U}_K=\{U_\alpha\times_{\Spec\mathbb{F}_1}\Spec K\}$ be a cover of $X_K$ coming from $\mathcal{U}$. Then we have
\[
\HH^1(X_K,\mathcal{O}_{X_K}^\times)=\HH^1(\mathcal{U}_K,\mathcal{O}_{X_K}^\times).
\] 
\end{cor}
\begin{proof}
Combine the previous result with Serre's version (\cite[Th\'{e}or\`{e}me 1 of $n^\circ$ $29$]{serre1955faisceaux}) of Leray's theorem.
\end{proof}

We now consider similar results for the cohomology of $\K\times\mathcal{O}_X^\times$. The following lemma is an analogue of Lemma \ref{coverofKM} for monoids. 

\begin{lem}\label{lemma: monoidcovering}
Let $M$ be a monoid and $\{U_\alpha\}$ be an open cover of $X=\Spec  M$.  Then $U_\alpha=X$ for some $\alpha$.
\end{lem}
\begin{proof}
One can observe that $M$ has a unique maximal ideal, namely, $\mathfrak{m}=(M-M^\times)$. Now, a similar argument as in Lemma \ref{coverofKM} proves our claim. 
\end{proof}

\begin{cor}
Let $M$ be a cancellative monoid and $X=\Spec M$. Then we have 
\[
\HH^k(X, \K\times \mathcal{O}_X^\times)=0, \quad \textrm{for } k\geq 1.
\]
\end{cor}
\begin{proof}
A similar argument to the one of Theorem \ref{thm: affine vanishing} yields the desired result. 
\end{proof}

\begin{cor}\label{F1acyclic}
Let $X$ be an irreducible monoid scheme and $\mathcal{U}=\{U_\alpha\}$ be an affine open cover of $X$ satisfying Condition \ref{condition: condition for open cover}. Then we have 
\[
\HH^1(X,\K\times\mathcal{O}_X^\times)=\HH^1(\mathcal{U},\K\times\mathcal{O}_X^\times).
\]
\end{cor}
\begin{proof}
Combine the previous result with Serre's version (\cite[Th\'{e}or\`{e}me 1 of $n^\circ$ $29$]{serre1955faisceaux}) of Leray's theorem.
\end{proof}

We are now able to express $\Pic X_K$ in terms of $X$.

\begin{prop}\label{3.14}
Let $X$ be an irreducible monoid scheme and $\mathcal{U}=\{U_\alpha\}$ an affine open cover satisfying Condition \ref{condition: condition for open cover}. Let $K$ be an idempotent semifield and let $X_K=X\times_{\Spec\mathbb{F}_1}\Spec K$.  Then we have 
\[
\Pic (X_K) \cong \HH^1(X,\K \times \mathcal{O}_X^\times).
\]
\end{prop}
\begin{proof}
Let $\mathcal{U}_K=\{U_\alpha\times_{\Spec\mathbb{F}_1}\Spec K\}$.  Then by Theorem \ref{almostmainthm} and Corollaries \ref{Kacylcic} and \ref{F1acyclic}, we have
\[
\Pic (X_K) \cong \HH^1(\mathcal{U}_K,\mathcal{O}_{X_K}^\times)\cong \HH^1(\mathcal{U},\K\times\mathcal{O}_X^\times)\cong  \HH^1(X,\K \times \mathcal{O}_X^\times).
\] 
\end{proof}

\begin{prop}\label{proposition: picard for monoid schemes}
Let $X$ be an irreducible monoid scheme.  Then we have
\[
\Pic (X)=\HH^1(X,\K\times\mathcal{O}_X^\times).
\]
\end{prop}
\begin{proof}
Since the constant sheaf on an irreducible space is flasque, we have that $\HH^1(X,\K)=0$. Then it follows that
\[
\HH^1(X,\K\times\mathcal{O}_X^\times)=\HH^1(X,\K)\times \HH^1(X,\mathcal{O}_X^\times)=0\times \Pic (X)=\Pic(X).
\] 
\end{proof}

Combining the two previous propositions gives the following theorem.

\begin{thm}\label{picthm}
Let $X$ be an irreducible monoid scheme and $\mathcal{U}=\{U_\alpha\}$ is an affine open cover satisfying Condition \ref{condition: condition for open cover}. Let $K$ be an idempotent semifield and let $X_K=X\times_{\Spec\mathbb{F}_1}\Spec K$.  Then 
\[
\Pic (X_K) \cong \Pic (X).
\]
\end{thm}

\begin{rmk}\label{rmk:Fulton}
We remark that the statements in this section still hold in the case when $K$ is not an idempotent semifield but rather the field of complex numbers $\mathbb{C}$, even though they may require different proofs. As mentioned in Remark~\ref{rmk: toricDeitmar} if $X$ is a connected integral monoid scheme of finite type then $X_{\mathbb{C}}$ is a toric variety. First, we observe that Condition \ref{condition: condition for open cover} holds for toric varieties. 

It is easy to see that Proposition \ref{KMunits} 
holds for $K = \mathbb{C}$. Proposition \ref{lem: open embedding} 
is true (via the extension of scalars functor). The proof of Theorem~\ref{almostmainthm} 
is the same over $\mathbb{C}$. Proposition~ \ref{thm: affine vanishing} 
is a classical result for affine (toric) varieties and implies the statement of Corollary \ref{Kacylcic}. 
Proposition~\ref{3.14} 
and Proposition~\ref{proposition: picard for monoid schemes} 
follow from the previous statements with the same proof and together imply Theorem~\ref{picthm}. 
This way we obtain that $\HH^1(X,\mathcal{O}_X^\times) = \HH^1(X_{\mathbb{C}},\mathcal{O}_{X_{\mathbb{C}}}^\times)$. 
This gives us a different proof of Theorem 6.6 of \cite{flores2014picard}. As observed in \cite{flores2014picard} the Cartier divisors on $X$ lift to torus-invariant Cartier divisors on $X_{\mathbb{C}}$ and this way we recover Fulton's result in \cite{Fulton} Section 3.4 that the Picard group of a toric variety is generated by torus invariant divisors.

\end{rmk}

\section{Cartier divisors on cancellative semiring schemes}
In this section, we define a Cartier divisor on a cancellative semiring scheme $X$, following the idea of Flores and Weibel \cite{flores2014picard}. We show that the Picard group $\Pic(X)$ is isomorphic to the group $\Cart(X)$ of Cartier divisors modulo principal Cartier divisors. In what follows, by an \emph{integral semiring}, we mean a semiring without zero divisors.\\
Let $A$ be an integral semiring and ${\mathfrak{p}} \in \Spec A$. We will call an element $f \in A^\times$ (multiplicatively) \emph{cancellable}, if the following condition holds:
\begin{equation}
af=bf \textrm{ implies } a=b, \quad \forall a, b \in A. 
\end{equation}
By a \emph{cancellative semiring}, we mean a semiring $A$ such that any nonzero element $a \in A$ is cancellable. 

\begin{mydef}\label{def:integral semiring scheme}
By an \emph{integral semiring scheme}, we mean a semiring scheme $X$ such that for any open subset $U$ of $X$, $\mathcal{O}_X(U)$ is an integral semiring. Moreover, if $\mathcal{O}_X(U)$ is cancellative for any open subset $U$ of $X$, we call $X$ a \emph{cancellative semiring scheme}. 
\end{mydef}

\begin{rmk}\label{rem: integral}
Classically, a nontrivial ring without zero divisors is multiplicatively cancellative and vice versa. However, a semiring without zero divisors is not necessarily cancellative - such as the polynomial semiring over the tropical semifield $\rma[x_1, \dots, x_n]$. Nonetheless, any nontrivial cancellative semiring is integral and hence any nontrivial cancellative semiring scheme is an integral semiring scheme. 
\end{rmk}

\begin{lem}\label{lem: unique generic point}
Let $X$ be an integral semiring scheme. Then $X$ has a unique generic point. 
\end{lem}
\begin{proof}
As in the classical case, if $X$ is an integral semiring scheme, then $X$ is irreducible and any irreducible topological space has a unique generic point. 
\end{proof}

\begin{lem}\label{lem; affinecase}
Let $X$ be an integral semiring scheme with a generic point $\eta$. Then for any affine open subset $U=\Spec A \subseteq X$, we have
\[
\mathcal{O}_{X,\eta} \simeq A_{(0)},
\]
where $A_{(0)}=\Frac(A)$. In particular, $\mathcal{O}_{X,\eta}$ is a semifield.
\end{lem}
\begin{proof}
Since $X$ is integral, $A$ is an integral semiring. In particular, $\mathfrak{p}=(0) \in \Spec A$ is the generic point of $\Spec A$. Again, since $X$ is integral, it follows that $\mathfrak{p}$ is also the generic point of $X$. Therefore we have $\mathcal{O}_{X,\eta}=A_{(0)}$. 
\end{proof}

\begin{mydef}\label{def:functionfield}
Let $X$ be an integral semiring scheme and $U=\Spec A$ be any affine open subset. We define the \emph{function field} $K(X)$ of $X$ as follows:
\[K(X)=\mathcal{O}_{X,\eta},\]
where $\eta$ is the generic point of $U$. 
\end{mydef}

\begin{myeg}\label{functionfieldeg}
Let $X=\Spec \rma[x]$, an affine line. Let $A=\rma[x]$. One can easily see that $A$ is an integral semiring and hence the generic point is $\eta=(0)$. Therefore, we have
\[K(X)=\mathcal{O}_{X,\eta}=\{ \frac{g}{f}\mid g \in A, f \in A\backslash \{0\}\}.
\]
\end{myeg}

\begin{prop}\label{prop: subsheaf}
Let $X$ be a cancellative semiring scheme with the function field $K$. Let $\mathcal{K}$ be the constant sheaf associated to $K$ on $X$. Then $\mathcal{O}_X^\times$ is a subsheaf (of abelian groups) of $\mathcal{K}^\times$. 
\end{prop}
\begin{proof}
First, suppose that $X$ is affine, i.e., $X=\Spec A$ for some cancellative semiring $A$. Notice that, since $A$ is cancellative, for any $f \in A\setminus\{0\}$, we have a canonical injection $i_f:A_f \to \Frac(A)$. Therefore, as in the classical case, for each open subset $U$ of $X$, we have
\begin{equation}
\mathcal{O}_X(U)=\bigcap_{D(f) \subseteq U}A_f \subseteq \Frac(A)=\mathcal{K}(U). 
\end{equation}
It follows that $\mathcal{O}_X^\times(U) \subseteq \mathcal{K}^\times(U)$. In general, one can cover $X$ with affine open subsets and the argument reduces to the case when $X$ is affine; this is essentially due to Lemma \ref{lem; affinecase}.  
\end{proof}

Thanks to Proposition \ref{prop: subsheaf}, we can define a Cartier divisor on a cancellative semiring scheme $X$ as follows:

\begin{mydef}\label{def:cartier}
Let $X$ be a cancellative semiring scheme with the function field $K$. Let $\mathcal{K}$ be the constant sheaf associated to $K$ on $X$. A \emph{Cartier divisor} on $X$ is a global section of the sheaf of abelian groups $\mathcal{K}^\times/\mathcal{O}_X^\times$.
\end{mydef}

We recall the notion of Cartier divisors on a cancellative monoid scheme. Let $X$ be a cancellative and irreducible monoid scheme with a generic point $\eta$. Denote by $K=\mathcal{O}_{X,\eta}$ the stalk at $\eta$ and by $\mathcal{K}$ be the associated constant sheaf.

\begin{mydef}\cite[\S 6]{flores2014picard}
\begin{enumerate}
\item 
A \emph{Cartier divisor} is a global section of the sheaf $\mathcal{K}^\times/\mathcal{O}_X^\times$ of abelian groups. We let $\Cart(X)$ be the group of Cartier divisors on $X$.
\item
A \emph{principal Cartier divisor} is a Cartier divisor which is represented by some $a \in \mathcal{O}_{X,\eta}^\times$. Let $P(X)$ be the subgroup of $\Cart(X)$ consisting of principal Cartier divisors. 
\item
We let $\CaCl(X)=\Cart(X)/P(X)$ be the group of Cartier divisors modulo principal Cartier divisors. 
\end{enumerate}
\end{mydef}

Flores and Weibel prove the following. 

\begin{mytheorem}\cite[Proposition 6.1.]{flores2014picard}\label{theorem: pic=car}
Let $X$ be a cancellative monoid scheme.  Then 
\[
\Pic(X) \simeq \CaCl(X). 
\]
\end{mytheorem}

As in the classical case, to any Cartier divisor $D$ on a cancellative semiring scheme $X$, one can associate an invertible sheaf $\mathcal{L}(D)$ on $X$. More precisely, if $D$ is a Cartier divisor represented by $\{U_i,f_i\}$, where $f_i \in \mathcal{O}_X^\times(U_i)$, then one defines an invertible sheaf $\mathcal{L}(D)$ as a subsheaf of $\mathcal{K}$ by requiring that $\mathcal{L}(D)(U_i)$ is generated by $f_i^{-1}$. In fact, the same argument as in the proof of Theorem \ref{theorem: pic=car} shows the following:

\begin{prop}\label{theorem: pic=car for semi}
Let $X$ be a cancellative semiring scheme. Then the map $\varphi:\CaCl(X) \to \Pic(X)$ sending $[D]$ to $[\mathcal{L}(D)]$ is an isomorphism, where $[D]$ is the equivalence class of $D \in \Cart(X)$ and $[\mathcal{L}(D)]$ is the equivalence class of $\mathcal{L}(D)$ in $\Pic(X)$.
\end{prop}
\begin{proof}
The argument is similar to \cite[Propoisiton 6.1]{flores2014picard}, however, we include a proof for the sake of completeness.\\
We have the following short exact sequence of sheaves of abelian groups:
\begin{equation}\label{exact sequence}
0 \longrightarrow \mathcal{O}_X^\times \longrightarrow \mathcal{K}^\times \longrightarrow \mathcal{K}^\times/\mathcal{O}_X^\times \longrightarrow 0.
\end{equation}
Since $X$ is cancellative, $X$ is irreducible. Furthermore, $\mathcal{K}^\times$ is a constant sheaf on $X$ and hence $\mathcal{K}^\times$ is flasque. In particular, $\HH^1(X,\mathcal{K}^\times)=0$. Therefore, the cohomology sequence induced by \eqref{exact sequence} becomes:
\begin{equation}
0 \longrightarrow \mathcal{O}_X^\times(X) \longrightarrow \mathcal{K}^\times(X) \overset{f}\longrightarrow (\mathcal{K}^\times/\mathcal{O}_X^\times)(X) \longrightarrow \HH^1(X,\mathcal{O}_X^\times) \longrightarrow 0. 
\end{equation}
But we have that $\HH^1(X,\mathcal{O}_X^\times)=\Pic(X)$,  $(\mathcal{K}^\times/\mathcal{O}_X^\times)(X)=\Cart(X)$, and $f(\mathcal{K}^\times(X))=P(X)$ and hence $\Pic(X) \simeq \CaCl(X)$ as claimed. 
\end{proof}

\begin{myeg}
Let $A=\rma[x_1,...,x_n]$ be the polynomial semiring over $\rma$. Recall that $A$ is integral, but not cancellative. Consider $B=A/\sim$, where $\sim$ is a congruence relation on $A$ such that $f(x_1,...,x_n) \sim g(x_1,...,x_n)$ if and only if $f$ and $g$ are same as functions on $\rma^n$. It is well--known (see for example \cite{jun2015cech} or \cite{KalinaDaniel1}) that in this case $B$ is cancellative. Now let $X=\Spec B$. It follows from Proposition \ref{theorem: pic=car for semi} that $\Pic(X)=\CaCl(X)$. But we know from \cite[Corollary 4.23.]{jun2015cech} that $\Pic(X)$ is the trivial group and hence so is $\CaCl(X)$. In tropical geometry, one may obtain a ``reduced model" of a tropical scheme as above, which could be used to compute $\CaCl(X)$, see \cite[Remark 4.27.]{jun2015cech}.
\end{myeg}

\begin{rmk}
One can easily see (cf. \cite{flores2014picard}) that if $X$ is a cancellative monoid scheme and $K$ is a field, then $\CaCl(X)$ is isomorphic to $\CaCl(X_K)$, where $\CaCl(X_K)$ is the group of Cartier divisors of the scheme $X_K$ modulo principal Cartier divisors. 
\end{rmk}

\begin{rmk} In the case of set-theoretic tropicalizations of curves, there is a well-developed theory of divisors.
Let $C$ be an algebraic curve defined over a valued field and $G$ be the set-theoretic tropicalization of $C$, or rather its non-Archimedian skeleton. We remind the reader, that $G$ is a graph. There exists a notion of a Picard group of $G$ (cf. \cite{baker2007riemann}). Moreover, there is a well-defined map $\Pic(C) \rightarrow \Pic(G)$, which one may think of as the tropicalization map. If $C$ has genus $0$ then the Picard groups $\Pic(C)$ and $\Pic(G)$ are equal and isomorphic to $\mathbb{Z}$. However, in general $\Pic(C)$ and $\Pic(G)$ are rarely the same. For example, consider an elliptic curve $E$ degenerating to a cycle of $\mathbb{P}^1$'s. Note that $\Pic^0(E) = E$ but $\Pic(G) = S^1$, the unit circle. In the few cases when the results of M. Baker and S. Norine are comparable with the results of this paper, that is when $X$ is a smooth toric curve, it can be verified explicitly that the Picard groups of the set-theoretic tropicalization (non-Archimedian skeleton) and the scheme-theoretic tropicalization of $X$ are the same.
\end{rmk}

\begin{myeg}[Computing the Picard group of the tropicalization of $\mathbb{P}^1 \times \mathbb{P}^1$]
{
We consider $X_{\mathbb{C}}=\mathbb{P}^1\times \mathbb{P}^1$, which is a toric variety. We compute explicitly the Picard group of the tropical scheme $X_{\rma}$. Note that the calculation is analogous to the classical one. \\

Let $\mathcal{U} = \cup_{i=1}^{4} U_i$ be an open cover for $X_{\rma} = (\mathbb{P}^1 \times \mathbb{P}^1)_{\rma}$, where
\begin{align*}
    U_1 &= \Spec \mathbb{T}[x, y], \\
    U_2 &= \Spec \mathbb{T}[x, y^{-1}], \\
    U_3 &= \Spec \mathbb{T}[x^{-1}, x^{-1}y^{-1}], \\
    U_4 &= \Spec \mathbb{T}[x^{-1}, xy].
\end{align*}
Now we can see which sections over each $U_i$ are units in $\Gamma(U, \mathcal{O}_{X_{\rma}})$, namely,
\begin{equation}\label{u_0}
A_i=\mathcal{O}_{X_{\rma}}^\times(U_i)=\mathbb{R}, \quad \forall i=1,2,3,4. 
\end{equation}
For instance, for $i=1$, we have that $\mathcal{O}_{X_{\rma}}^\times(U_1)$ is the group of (tropically) multiplicatively invertible elements in $\mathcal{O}_{X_{\rma}}(U_1)=\rma[x,y]$. Since the only tropical polynomials, which are multiplicatively invertible, are nonzero constants, we have \eqref{u_0}. \\

Let $U_{ij}=U_i \cap U_j$. Now we have
\[
A_{12}= \mathcal{O}_{X_{\rma}}^\times(U_{12})=(\rma[x, y^{\pm 1}])^\times \cong \mathbb{R} \times \mathbb{Z},\] and similarly $A_{14}$, $A_{23}$, $A_{34}$ are isomorphic to $\mathbb{R} \times \mathbb{Z}$. Note, however, that
\[
A_{13} \cong A_{24} = \rma ([x^{\pm 1}, y^{\pm 1}])^\times \cong \mathbb{R} \times \mathbb{Z}^2.
\]
Similarly for $U_{ijk}=U_i \cap U_j \cap U_k$ one may also check that:
\[
A_{ijk}=\mathcal{O}_{X_{\rma}}^\times(U_{ijk})=(\rma[x^{\pm 1}, y^{\pm 1}])^\times\cong \mathbb{R} \times \mathbb{Z}^2.
\]
In particular, we get that 
\begin{equation}
\check{C}^0(X_{\rma},\mathcal{O}_{X_{\rma}}^\times) \cong \mathbb{R}^4, \quad \check{C}^1(X_{\rma}, \mathcal{O}_{X_{\rma}}^\times) \cong (\mathbb{R}\times\mathbb{Z})^4 \times (\mathbb{R}\times \mathbb{Z}^2)^2, \quad \check{C}^2(X_{\rma},\mathcal{O}_{X_{\rma}}^\times) \cong (\mathbb{R}\times \mathbb{Z}^2)^4.
\end{equation}

We can proceed with the computation using the usual \v{C}ech complex or we can do a ``tropical" computation as in \cite{jun2015cech} considering the following cochain complex:
\begin{equation}\label{cech complex}
\xymatrixcolsep{3pc}\xymatrix{
\ar[r]_-{d_{0}^-} \ar@<1ex>[r]^-{d_{0}^+} 
&
 A_{12} \times A_{13} \times A_{23} \times A_{14} \times A_{24} \times A_{34} \ar[r]_-{d_{1}^-} \ar@<1ex>[r]^-{d_{1}^+} 
&  A_{123} \times A_{124} \times A_{134} \times A_{234}  \ar[r]_-{d_{2}^-} \ar@<1ex>[r]^-{d_{2}^+}
& \cdots }
\end{equation}

The complex \eqref{cech complex} was introduced to deal with the lack of subtraction in tropical geometry; instead of one differential we consider pairs of morphisms $(d_{i}^+, d_{i}^-)$. Also, instead of the difference of two functions $f-g$, we will have a tuple $(f, g)$ and we replace the kernel condition $f-g = 0$ with $(f, g) \in \Delta$, where $\Delta$ is the diagonal. 

However, since $\mathcal{O}_X^\times$ is a sheaf of abelian groups, we can simply use the classical cochain complex. \\
Either way, we can see that the image of $d_0$ is generated by elements of the form $(f_1, f_2, f_3, f_4)$, where $f_i \in \mathbb{R}$ for all $i$ and $(f_1, f_2, f_3, f_4) \neq \lambda (1, 1, 1, 1)$ for some $\lambda \neq 0$ and thus the image of $d_0$ is isomorphic to $\mathbb{R}^3$. The kernel of $d_1$ is generated by elements of the following form:
\begin{equation}\label{forsm of generators}
(\frac{b}{c}y^k, bx^ly^k, cx^k, ax^l, \frac{ac}{b}x^ly^k, \frac{a}{b}y^k),\quad \textrm{for }a, b, c \in \mathbb{R}\textrm{ and } k, l \in \mathbb{Z}.
\end{equation}
Now, for the choice of $(f_1,f_2,f_3,f_4)=(b,c,1_\mathbb{T},\frac{b}{c})$, one can easily see that any element as in \eqref{forsm of generators} defines the same equivalence class as an element $(y^k, x^ly^k, x^k, x^l, x^ly^k, y^k)$ in $\cHH^1(X_{\rma}, \mathcal{O}_{X_{\rma}}^\times)$. Now it is easy to see that 
\[
\HH^1(X_{\rma},  \mathcal{O}_{X_{\rma}}^\times) = \cHH^1(X_{\rma}, \mathcal{O}_{X_{\rma}}^\times) \cong \mathbb{Z}^2.
\]

}
\end{myeg}

\begin{rmk}
In the above computation the fact that $X$ is a product of projective lines is not crucial. We can identify $\mathbb{P}^1 \times \mathbb{P}^1$  with its image under the Segre embedding into $\mathbb{P}^3$, i.e., $X_{\rma}$ is defined by a congruence generated by a single ``bend relation" $x_0x_1 \sim x_2x_3$. Both computations give the same result as expected.
\end{rmk}

\bibliography{picard}\bibliographystyle{plain}

\end{document}